\documentclass[a4paper,fleqn,10pt,twocolumn]{SICE_ISCS}
\usepackage{amsmath,amssymb}
\usepackage{color}
\usepackage{graphicx}
\usepackage[export]{adjustbox}
\usepackage{url}
\usepackage{comment}

\title{Stabilization of hyperbolic reaction-diffusion systems on directed networks through the generalized Routh-Hurwitz criterion for complex polynomials}

\author{Riccardo Muolo${}^{1\dagger}$, Anthony Hastir${}^{2,3}$ and Hiroya Nakao${}^{1}$}
\speaker{Riccardo Muolo}

\affils{${}^{1}$Department of Systems and Control Engineering, Tokyo Institute of Technology, Japan\\
(Tel: +81-3-5734-3216; E-mail: muolo.r.aa@m.titec.ac.jp, nakao@sc.e.titech.ac.jp)\\
${}^{2}$Department of Mathematics and Computer Science, Bergische Universität Wuppertal, Germany\\
${}^{3}$Department of Mathematics and naXys, Namur Institute for Complex Systems, University of Namur, Belgium\\
(Tel: +32-8-172-4940.; E-mail: anthony.hastir@unamur.be)\\
}
\abstract{%
The study of dynamical systems on complex networks is of paramount importance in engineering, given that many natural and artificial systems find a natural embedding on discrete topologies. For instance, power grids, chemical reactors and the brain, to name a few, can be modeled through the network formalism by considering elementary units coupled via the links. In recent years, scholars have developed numerical and theoretical tools to study the stability of those coupled systems when subjected to perturbations. In such framework, it was found that asymmetric couplings enhance the possibilities for such systems to become unstable. Moreover, in this scenario the polynomials whose stability needs to be studied bear complex coefficients, which makes the analysis more difficult. In this work, we put to use a recent extension of the well-known Routh-Hurwitz stability criterion, allowing to treat the complex coefficient case. Then, using the Brusselator model as a case study, we discuss the stability conditions and the regions of parameters when the networked system remains stable.}


\begin{document}

\maketitle


\section{Introduction}

Studying stability of a dynamical system is a paramount step before designing appropriate control laws in engineering. For linear differential equations, this task may be carried out by looking at the eigenvalues of the matrix describing the dynamics. When moving to the nonlinear setting, the eigenvalues of the Jacobian matrix of the vector field also give information, but only on the local stability of a particular equilibrium, as shown by the Hartman-Grobman Theorem for the case of fixed point equilibria, or the Andronov-Witt theorem for limit cycle equilibria. Similar results exist when the state-space is infinite-dimensional, typically a Banach or a Hilbert space. This setting is particularly useful to study partial differential equations (PDEs) in an abstract way. In this context, appropriate definitions of linearization have to be given in order to link the stability of a linearized system and its nonlinear version. These questions have been recently analyzed in \cite{Hastir_TAC} and \cite{Hastir_Automatica} in order to study the stability of the equilibria of a nonisothermal axial dispersion tubular reactor based on a new version of the so-called "Fr\'echet differentiability". Therein, it was notably shown that the multiplicity and the stability of the equilibria are "diffusion driven" in the sense that the diffusion coefficient in the process determines whether one or three equilibria are exhibited. Moreover, in the case of only one equilibrium, the latter has been shown locally exponentially stable for the nonlinear system while bistability is depicted in the case of three equilibria, that is, stable -- unstable -- stable.\\

In such context, the goal is to annihilate the formation of instabilities, which disrupt the homogeneous equilibrium of our system. However, when discussing reaction-diffusion system in a stationary state, the objective may as well be antagonistic, as scholars often look for ways to obtain Turing instabilities, which are consequence of a diffusion-driven mechanism, i.e., caused by a difference in the diffusivities of the two interacting species \cite{Turing}, or other kinds of instabilities. Turing theory was conceived in the framework of (parabolic) PDEs, but the setting has then been extended to lattices by Othmer and Scriven \cite{OS1971} and, successively, on complex networks by Nakao and Mikhailov \cite{NM2010}. The formalism developed in the latter work enabled the study of asymmetric couplings, where it was shown that such asymmetry enhanced the diffusion-driven mechanism of instability \cite{Asllani1}. In fact, it was then generally proved that non-reciprocal interactions enhance inhomogeneities \cite{carletti}. 

When dealing with a more realistic setting, an upper bound to the signal propagation speed is considered and one obtains hyperbolic reaction-diffusion system. Now, besides the "classical" diffusion-driven instability, there is a whole new class of patterns caused by the inertia in the signal propagation, as shown by Zemskov and Horsthemke \cite{ZH2016}. The latter study has then been extended to networked systems \cite{jop_carletti}, but with symmetric coupling. In this work, we make a step further by studying the case of hyperbolic reaction-diffusion systems on directed networks. From the above discussion, we can already understand that in such setting it is very difficult to achieve stability: in fact, diffusion, inertia and the asymmetry of the coupling are all factors conspiring against the stability of the homogeneous equilibrium. Moreover, the fact of dealing with an asymmetric coupling has the consequence that the coupling operator has a complex spectrum, making the usual stability analysis techniques, such as the Routh-Hurwitz criterion \cite{Routh1877,Hurwitz1895}, not applicable. Hence, we put to use a recent generalization of such techniques, developed to deal with complex polynomials \cite{HastirMuolo_RH_Gen}. This allows us to visualize the stability regions, which can be fundamental in the design of systems where stability is fundamental, such as chemical reactors arrays. \\

The paper is structured as follows: in the next section we will discuss the linear stability analysis for hyperbolic reaction-diffusion systems on networks; then, in Sec. \ref{sec:3}, we will put to use the extended Routh-Hurwitz criterion for complex polynomials and analyze the stability conditions. Finally, in Sec. \ref{sec:4}, we will aid the visualization of the stability regions through numerical simulations for the case of a chemical model, the Brusselator, and, in Sec. \ref{sec:conclusion}, we conclude by discussing some \textit{caveats} and limitations of the linear stability analysis.

\section{Stability of hyperbolic reaction-diffusion systems on directed networks}\label{sec:2}

Let us consider a network on $n$ nodes, defined by its adjacency matrix $A$, whose elements $A_{ij}$ take value $1$ if there is a link between nodes $i$ and $j$, and $0$ otherwise. A diffusion process on such network is modeled through the discrete diffusion equation. For instance, if we consider the diffusion of a chemical species $u$, we have, for each node $i$ \begin{equation}
    \dot{u}_i=D_u\sum_{i=1}^{n}L_{ij}u_j,
\label{eq:disc_diff}
\end{equation}
where $D_u$ is the diffusion coefficient of species $u$, whose dimensions are $[m^2/s]$ and $L_{ij}=A_{ij}-\delta_{ij}k_i$ is the discrete Laplacian matrix, a negative semi-definite matrix playing the role of the diffusion operator on networks \cite{latora_nicosia_russo_2017}. $k_i$ is the node degree, i.e., to how many other nodes it is connected. The spectrum of the discrete Laplacian is denoted by $\{\Lambda^\alpha \}_{\alpha\in\{1,...,n\}}$, where $\Lambda^1=0$ {is the Laplacian's spectral abscissa, i.e., the maximum real part of the spectrum}. If the network is connected, then we have that $\Lambda^1>\Lambda^\alpha$ $\forall \alpha\in\{2,...,n\}$. Let us remark that, the network being symmetric, the spectrum of the Laplacian is real.\\
The above description works well when the domain where the dynamics takes place (e.g., the network) is small or the diffusivity $D_u$ is large, i.e., one can consider the signal propagation as instantaneous. However, in many practical applications this approximation does not work, hence the diffusion equation needs to be corrected by accounting for an inertia (hence, with dimension $[1/s]$) in the signal propagation, generally indicated with $\tau$, as done by Cattaneo in the 50s \cite{Cattaneo1958}. Eq. \eqref{eq:disc_diff} then becomes
\begin{equation}
    \tau_u\ddot{u}_i+\dot{u}_i=D_u\sum_{i=1}^{n}L_{ij}u_j.
\end{equation}\label{eq:disc_diff_hyp}
The above equation takes several names (Cattaneo equation, telegrapher's equation, relativistic heat equation, to name a few) and it is usually studied in the framework of PDEs, where it is a hyperbolic equation. It has been rigorously derived on networks in \cite{jop_carletti}, where the name \textit{hyperbolic} has been kept despite being an ODE in such framework.\\

Let us now consider a system of two species, $u$ and $v$, which interact by means of two nonlinear functions $f$ and $g$. We consider $n$ identical copies of such system, whose dynamics is described by the following \begin{displaymath}
    \begin{cases}
        \dot{u}_i=f(u_i,v_i), \\ \dot{v}_i=g(u_i,v_i).
    \end{cases}
\end{displaymath} The above system could describe, for instance, the dynamics into a well-stirred environment, such as a chemical reactor. Such dynamics could be, for instance, a stationary state or an oscillating one. Let us observe that, since we are considering a homogeneous domain, there is no propagation, hence the second time derivatives do not appear. Let us know study the case of $n$ copies of the above system and then couple them, by letting the chemical concentrations $u_i$ and $v_i$ diffuse between the reactor in a Fickean (linear) fashion. We are, hence, assuming that the dynamics take place within the nodes, while the links are the means by which the particles can move between adjacent systems. Such setting is schematized in Fig. \ref{fig:fig0}.

\begin{figure}[ht!]
\includegraphics[scale=0.12]{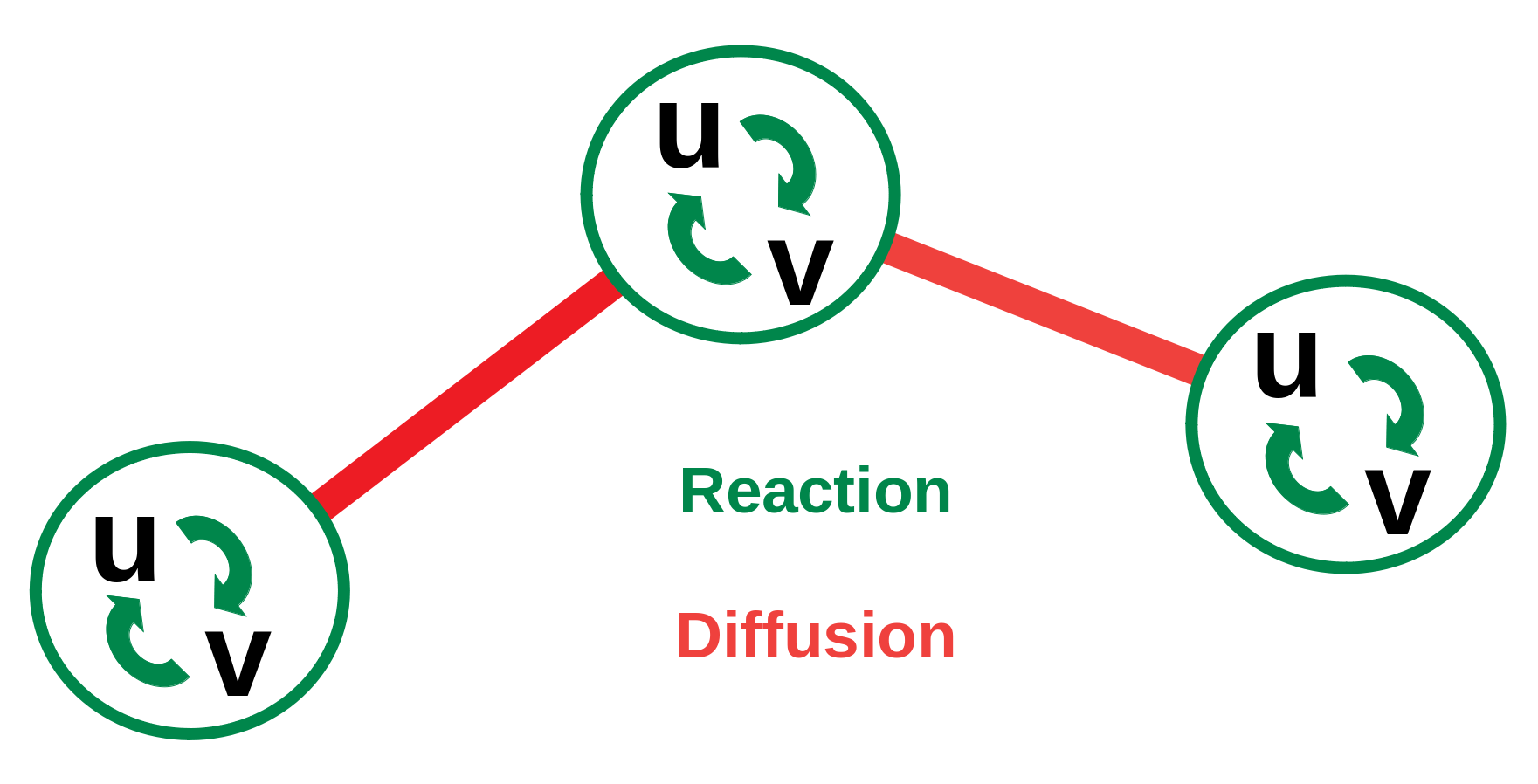}
\caption{The interaction between the two species takes place within the nodes, while the diffusion is through the links. In such setting the reaction-diffusion dynamics is described by a system of ODEs, because the dynamics within the nodes is considered to be a mean-field process, while the space is identified by the network nodes.}\label{fig:fig0}
\end{figure}

If we also account for a finite propagation, we obtain the following hyperbolic reaction-diffusion networked system: 

\begin{equation}
 \begin{cases}  \displaystyle \tau_u\ddot{u}_i+\dot{u}_i=f(u_i,v_i)+D_u\sum_{i=1}^{n}L_{ij}u_j, \\ \\ \displaystyle \tau_v\ddot{v}_i+\dot{v}_i=g(u_i,v_i)+D_v\sum_{i=1}^{n}L_{ij}v_j.
 \end{cases}
\end{equation}\label{eq:hyp_sys3}
As explained in the introduction, let us stress that the described framework is general, and can describe different systems, from chemical systems to ecological neural ones, to name a few. \\

Let us now discuss the stability of system (3) . For sake of simplicity, let us assume that each isolated system lies in a stationary state $(u^*,v^*)$. {Such equilibrium point can be found by solving the system $f=0,~g=0$.} The case of oscillating, or even chaotic, dynamics can be treated with the same techniques, but one cannot obtain analytical results and one must resort to numerical tools, such as the Master Stability Function \cite{pecora1998master}. The first step is to impose the stability of the isolated system, meaning that the Jacobian matrix computed at the stationary state needs to be stable, i.e., given 
\begin{displaymath}
    J_0=\left[\begin{matrix} 
f_u & f_v \\ g_u & g_v
\end{matrix}\right],
\end{displaymath} where $f_u$ is the derivative of the function $f$ with respect to the variable $u$ computed in $(u^*,v^*)$, we need the trace to be negative and the determinant to be positive, i.e, \begin{displaymath}
\begin{cases}
    f_u+g_v<0 ,\\ f_ug_v -f_vg_u>0.
    \end{cases}
\end{displaymath} 
Let us point out that, in general, the inertia in the signal propagation can affect the stability of the homogeneous equilibrium \cite{jop_carletti}. In what follows, however, we will consider settings in which the stability is not affected by the inertia. One can then study whether an inhomogeneous perturbation about such stable homogeneous equilibrium disrupts it, giving rise to an instability, or goes back to the stationary state. The full discussion of the case in which the network is symmetric can be found in \cite{jop_carletti}. In this work, we want to extend such framework by considering asymmetric couplings, which are more common in practical applications. In this case, the equation describing the dynamical process are the same as Eq. \eqref{eq:hyp_sys3}, but with the important difference that now the Laplacian is a directed matrix, defined as $L_{ij}=A_{ij}-\delta_{ij}k_i^{in}$, where $k_i^{in}$ is the in-degree, i.e., how many directed links are incoming to node $i$. Hence, its spectrum will be, in general, complex, meaning that now $\Lambda^\alpha=\Lambda_{Re}^\alpha+i\Lambda_{Im}^\alpha$ $\forall \alpha\in\{2,...,n\}$, {while $\Lambda^1=0$ remains the largest eigenvalue, i.e., $\Lambda_{Re}^\alpha<0=\Lambda^1$ for any $\alpha\in\{2,...,n\}$.}\\

The linear stability analysis breaks down to compute \begin{equation}
\det \left[\begin{matrix} 
J_{11} & J_{12} \\ J_{21} & J_{22}
\end{matrix}\right]=0,
\label{Det_Zero_J}
\end{equation} 
with
\begin{displaymath}
\begin{cases}
J_{11}= \tau_u\lambda^2+\lambda -f_u-D_u \Lambda_{Re}^{\alpha}-i D_u \Lambda_{Im}^{\alpha},\\
J_{12}=-f_v, \\
J_{21}= -g_u, \\
J_{22} = \tau_v\lambda^2+\lambda -g_v-D_v \Lambda_{Re}^{\alpha}-i D_v \Lambda_{Im}^{\alpha},
\end{cases}
\end{displaymath}
{where $\lambda$ is the growth rate of the perturbation in the linear regime, i.e., if it is negative, the system will go back to its stable state, while if it is positive, there will be an exponential instability driving the system away from the equilibrium.}

From the above expressions, \eqref{Det_Zero_J} is equivalent to the following equation  
\begin{equation}
    \lambda^4+\sum_{j=1}^4 (a_j+ib_j)\lambda^{4-j}=0,
    \label{Main_Poly}
\end{equation} 
whose coefficients $a_j, b_j, j=1,2,3,4$, are given by
\begin{displaymath}
\begin{cases}
\displaystyle a_1=\varepsilon(\tau_u+\tau_v),
\\[7pt]
\displaystyle b_1=0,\\[3pt]
\displaystyle a_2=\varepsilon\Big(1-g_v\tau_u-f_u\tau_v-(\tau_u D_v+\tau_v D_u)\Lambda_{Re}^{\alpha}\Big),\\[7pt]
 \displaystyle b_2=\varepsilon\Big(-\Lambda_{Im}^{\alpha}(\tau_u D_v+\tau_v D_u)\Big),\\[7pt]
\displaystyle a_3=\varepsilon\Big(-g_v+f_u+(D_u+D_v)\Lambda_{Re}^{\alpha}\Big),\\[7pt]
\displaystyle b_3=\varepsilon\Big(-\Lambda_{Im}^{\alpha}(D_u+D_v)\Big),\\[7pt]
\displaystyle a_4=\varepsilon\Big[(f_uD_v +g_vD_u)\Lambda_{Re}^{\alpha}+f_u g_v-f_v g_u \\
~~~~~~~~~~~~~~~~~~~~~~~~~~~~~+D_u D_v\Big((\Lambda_{Re}^{\alpha})^2+(\Lambda_{Im}^{\alpha})^2\Big)\Big],\\
\displaystyle b_4=\varepsilon\Big((g_vD_u+f_uD_v)\Lambda_{Im}^{\alpha}+D_uD_v\Lambda_{Re}^{\alpha}\Lambda_{Im}^{\alpha}\Big),
\end{cases}
\end{displaymath}
with $\varepsilon := (\tau_u\tau_v)^{-1}$. The stability of the above polynomial cannot be studied through the Routh-Hurwitz criterion \cite{Routh1877,Hurwitz1895}, which can be applied only to polynomials with real coefficients. However, we can rely on a recent generalization to the complex case, as we discuss in the next section.

\section{The generalized Routh-Hurwitz criterion}\label{sec:3}

This section is devoted to the characterization of stability/instability of the polynomial \eqref{Main_Poly}. To this end, one shall rely on a generalization of the classical Routh-Hurwitz criterion for polynomials with complex coefficients. This method has been initiated in \cite{Frank_1946} and recently presented in an algorithmic way in \cite{HastirMuolo_RH_Gen}. The main advantage of \cite{HastirMuolo_RH_Gen} is the ease of implementation of the method as a pedagogical tool. From an algorithmic point of view, the method presented as in \cite{HastirMuolo_RH_Gen} does not involve more than determinants of $2\times 2$-matrices at each step, no matter the degree of the polynomial.  The idea of the proposed algorithm consists in building a table of coefficients similar to the one introduced by Hurwitz in \cite{Hurwitz1895} and then to exploit its first column to characterize stability/instability. For a 4-th order polynomial of the form \eqref{Main_Poly}, the table constructed via \cite[Algorithm 1]{HastirMuolo_RH_Gen} is given in Tab. \ref{tab:Table_RH_Gen_4}.

\begin{table*}
\begin{center}
\begin{tabular}{ |p{3.5cm}|p{3.5cm}|p{3.5cm}|p{1.8cm}|}
\hline
$a_1^{(1)}=~a_1$  & $b_2^{(1)}=~b_2$ & $a_3^{(1)}=~a_3$ & $b_4^{(1)}=~b_4$ \\
$b_1^{(1)}=a_1b_1-b_2$ & $a_2^{(1)}=a_1a_2-a_3$ & $b_3^{(1)}=a_1b_3-b_4$ & $a_4^{(1)}=a_1a_4$\\
\hline
$a_2^{(2)}=a_1a_2^{(1)}+b_1^{(1)}b_2$ & $b_3^{(2)}=a_1b_3^{(1)}-b_1^{(1)}a_3$ & $a_4^{(2)}=a_1a_4^{(1)}+b_1^{(1)}b_4$ &\\
$b_2^{(2)}=a_2^{(2)}b_2-a_1b_3^{(2)}$ & $a_3^{(2)}=a_2^{(2)}a_3-a_1a_4^{(2)}$ & $b_4^{(2)}=a_2^{(2)}b_4$ &\\
\hline
$a_3^{(3)}=a_2^{(2)}a_3^{(2)}+b_2^{(2)}b_3^{(2)}$ & $b_4^{(3)}=a_2^{(2)}b_4^{(2)}-b_2^{(2)}a_4^{(2)}$ & &\\
$b_3^{(3)}=a_3^{(3)}b_3^{(2)}-a_2^{(2)}b_4^{(3)}$ & $a_4^{(3)}=a_3^{(3)}a_4^{(2)}$ & &\\
\hline
$a_4^{(4)}=a_3^{(3)}a_4^{(3)}+b_3^{(3)}b_4^{(3)}$ & & &\\
\hline
\end{tabular}
\caption{Coefficients of the generalized Routh-Hurwitz criterion}\label{tab:Table_RH_Gen_4}
\end{center}
\end{table*}
As a consequence, the stability/instability of \eqref{Main_Poly} is characterized in the following proposition.

\begin{prop}
The polynomial \eqref{Main_Poly} is stable if and only if the following three conditions are satisfied, namely
\begin{align}
    &\upsilon := \varepsilon\left(\tau_u + \tau_v\right)\left[\tau_u+\tau_v - g_v\tau_u^2 - f_u\tau_v^2 - \tau_u^2 D_v\Lambda_{Re}^\alpha\right.\nonumber\\
    &\left.- \tau_v^2 D_u\Lambda_{Re}^\alpha\right] - \Lambda_{Im}^\alpha\left(D_u+D_v\right)^2 > 0\label{Condi_One}\\
    &\gamma := \alpha(a_3\alpha - a_1^3 a_4 + a_1 b_2 b_4) + \left[\alpha b_2 - a_1\beta\right]\beta > 0\label{Condi_Two}\\
    &\left(\gamma\beta-\alpha\right)\left[\alpha^2 b_4 - (a_1^2 b_4 - b_2 b_4)(\alpha b_2 - a_1\beta)\right]\nonumber\\
    &+ (a_1^2 a_4 - b_2 b_4)\gamma^2 > 0\label{Condi_Three},
\end{align}
where $\alpha := \varepsilon^2\upsilon$ and $\beta := a_1(a_1 b_3 - b_4) + a_3 b_2$.
\end{prop}
\begin{proof}
First, observe that, thanks to \cite[Theorem 2]{HastirMuolo_RH_Gen}, the polynomial \eqref{Main_Poly} is stable if and only if the coefficients $a_k^{(k)}$ in Tab. \ref{tab:Table_RH_Gen_4} are all positive.
Then, note that the inertia constants $\tau_u$ and $\tau_v$ are positive, which makes condition $a_{1}^{(1)}>0$ trivially satisfied. By looking at the constant $a_2^{(2)}$, there holds
\begin{align*}
    &a_2^{(2)} = a_1(a_1 a_2-a_3)-b_2^2\\
    &= \left(\frac{1}{\tau_u}+\frac{1}{\tau_v}\right)^2\varepsilon\left[1-g_v\tau_u-f_u\tau_v\right.\\
    &\left.-(\tau_u D_v+\tau_v D_u)\Lambda_{Re}^\alpha\right] + \left(\frac{1}{\tau_u}+\frac{1}{\tau_v}\right)\varepsilon\left[g_v+f_u\right.\\
    &\left.+ (D_u+D_v)\Lambda_{Re}^\alpha\right] - \varepsilon^2\left[\Lambda_{Im}^\alpha(D_u+D_v)\right]^2\\
    &=\left(\frac{1}{\tau_u}+\frac{1}{\tau_v}\right)\varepsilon\left[\left(\frac{1}{\tau_u}+\frac{1}{\tau_v}\right) - g_v\frac{\tau_u}{\tau_v} - f_u\frac{\tau_v}{\tau_u}\right.\\
    &\left.- \frac{\tau_u}{\tau_v}D_v\Lambda_{Re}^\alpha - \frac{\tau_v}{\tau_u}D_u\Lambda_{Re}^\alpha\right] - \varepsilon^2\Lambda_{Im}^\alpha(D_u+D_v)^2\\
    &=\varepsilon^2\upsilon,
\end{align*}
where $\upsilon$ is defined in \eqref{Condi_One}. Then, $a_2^{(2)}> 0$ if and only if $\upsilon >0$, which shows \eqref{Condi_One}.  Let us now move to the coefficient $a_3^{(3)}$. One has that 
\begin{align*}
    &a_3^{(3)} = \alpha (\alpha a_3 - a_1 a_4^{(2)}) + (\alpha b_2 - a_1 b_3^{(2)})b_3^{(2)}\\
    &= a_3\alpha^2 - a_1\alpha(a_1 a_4^{(1)} - b_2 b_4) + b_2\alpha(a_1 b_3^{(1)}+b_2 a_3)\\
    & - a_1(a_1 b_3^{(1)} + b_2 a_3)^2.
\end{align*}
The condition \eqref{Condi_Two} then follows from the equalities $a_4^{(1)} = a_1a_4$ and $b_3^{(1)} = a_1b_3-b_4$. The condition \eqref{Condi_Three} may be obtained in a similar way, by expressing $a_4^{(4)}$ as a function of $a_3^{(3)} = \gamma$ and $a_2^{(2)} = \varepsilon^2\upsilon$.
\end{proof}

\section{Numerical results}\label{sec:4}

As stated in the introduction, our goal is to annihilate the formation of instabilities, which disrupt the homogeneous equilibrium of our system. In what follows, we will focus our attention to instability arising from the inertia in the signal propagation, or inertia-driven instabilities. In fact, the case of diffusion-driven instabilities has been thoroughly studied in past works, even though not specifically in the context of asymmetrically coupled hyperbolic reaction-diffusion systems, while inertia-driven instabilities constitute a rather unexplored setting; a gap which we aim to contribute in filling with the present work. \\

As we have shown in the previous section, we can find explicit expressions for the stability conditions thanks to the generalized Routh-Hurwitz criterion. Nonetheless, such expressions remain cumbersome and difficult to handle analytically and the stability region is difficult to visualize. Hence, we resort to numerical simulation to validate the developed theory. For this, we will make use of the Brusselator, a model describing the dynamics of an autocatalytic reaction, paradigmatic in the study of chemical instabilities and synchronization in reaction diffusion-systems \cite{PrigogineNicolis1967}. The version of the model that we consider \cite{galla} models the nonlinear interactions between the two chemical species $u$ and $v$ by mean of two positive parameters $b$ and $c$, namely \begin{equation}
\begin{cases}
    f(u,v)=1-(b-1)u+cu^2v, \\
    g(u,v)=bu-cu^2v.
    \end{cases}
\end{equation}\label{eq:bruss}

The above system has a unique fixed point in \begin{displaymath}
(u^*=1,v^*=b/c), \end{displaymath} hence, the Jacobian matrix becomes

\begin{displaymath}
    J_0=\left[\begin{matrix} b-1 & c \\ -b & -c
\end{matrix}\right].
\end{displaymath}

Now, if we consider the above model on an asymmetric networks of $n$ nodes, diffusively coupled, and with a finite signal propagation, we obtain the following hyperbolic reaction-diffusion networked system
\begin{equation}
 \begin{cases}  \displaystyle \tau_u\ddot{u}_i+\dot{u}_i=1-(b-1)u_i+cu_i^2v_i\\
 \displaystyle\hspace{2.5cm}+D_u\sum_{i=1}^{n}L_{ij}u_j, \\ \displaystyle \tau_v\ddot{v}_i+\dot{v}_i=bu_i-cu_i^2v_i+D_v\sum_{i=1}^{n}L_{ij}v_j.
 \end{cases}
 \label{eq:hyp_bruss}
\end{equation}

If we perform a linear stability analysis, as in Sec. \ref{sec:2}, and then apply the generalized Routh-Hurwitz criterion, as in Sec. \ref{sec:4}, we obtain the stability condition for system \eqref{eq:hyp_bruss}. As already mentioned, we want to focus on the instability driven by the inertia in the signal propagation, hence let us fix the diffusion coefficients and the model parameters and focus on the inertia terms $\tau_u$ and $\tau_v$. The other free parameters will be the real and imaginary parts spectrum of the Laplacian, i.e., $\Lambda_{Re}^{\alpha}$ and $\Lambda_{Im}^{\alpha}$, which will be considered as continuous parameters. In fact, in principle, the spectrum of the coupling operator could be anywhere in $\mathbb{C}^-$, but then there will be only $n$ eigenvalues, i.e., the size of the network. Except for $\Lambda^{1}=0$, all the other eigenvalues are potential instability modes and need to be controlled onto the stability region, as can be visualized in Fig. \ref{fig:fig1}.

\begin{figure}[ht!]
\includegraphics[scale=0.55]{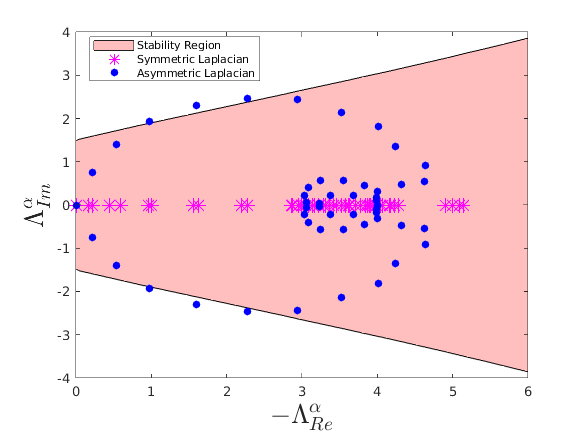}
\caption{Stability regions of the polynomial \eqref{Main_Poly} and spectra of a symmetric and an asymmetric (directed) Laplacians. We see that the spectrum of the symmetric Laplacian (magenta stars) always lie inside the stability region, while the imaginary part of the directed Laplacian spectrum (blue dots) reach the instability region. The directed network has $50$ nodes and it is generated with the Newman-Watts algorithm, as in \cite{Asllani1}, and the undirected one is obtained by symmetrizing the adjacency matrix of the former, i.e., $A_{sym}=(A+A^T)/2$. The parameters are $b = 1.3$, $c = 14$, $D_u = 0.5$, $D_v = 0.5$, $\tau_u=2$ and $\tau_v=1$.}\label{fig:fig1}
\end{figure}

The numerical results can be visualized in Fig. \ref{fig:3D_Surf_Stab_Reg}, where we plot the stability region for system \eqref{eq:hyp_bruss}, i.e., the region where the conditions \eqref{Condi_One}, \eqref{Condi_Two} and \eqref{Condi_Three} are satisfied. Let us observe that, for this choice of parameters in which the two diffusion coefficients are equivalent, i.e., $D_u\equiv D_v$, the only factors causing the instability are the presence of the inertia terms $\tau_u$ and $\tau_v$ and the imaginary part of the spectrum $\Lambda_{Im}^{\alpha}$, induced by the asymmetry of the coupling. In fact, diffusion driven instabilities arise only when $D_u<D_v$ \cite{NM2010}.

\begin{figure*}[ht!]
\includegraphics[scale=1]{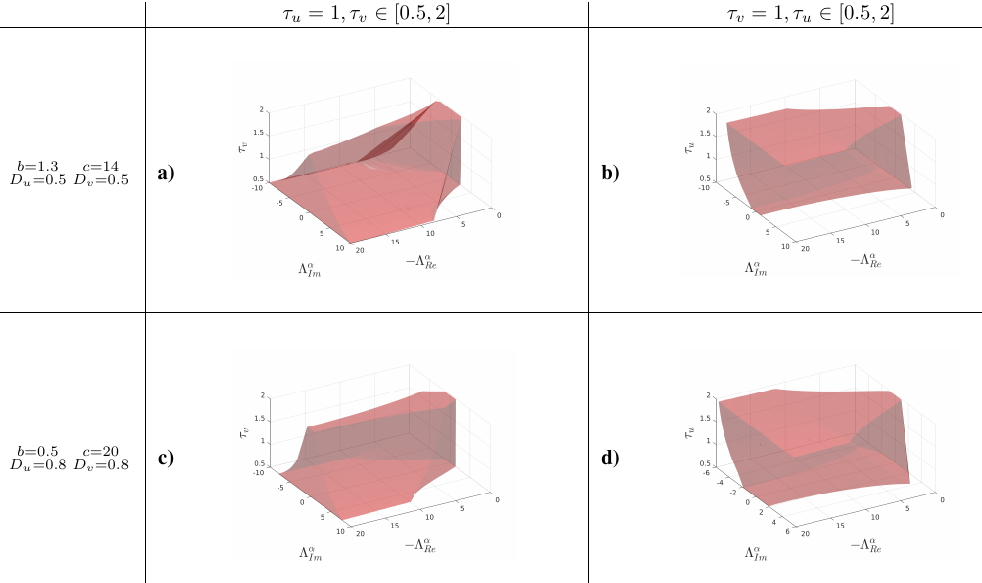}
\caption{Stability regions of the polynomial \eqref{Main_Poly} for different configurations of the parameters $b, c, \tau_u, \tau_v, D_u$ and $D_v$.\label{fig:3D_Surf_Stab_Reg}}
\end{figure*}

In Fig. \ref{fig:3D_Surf_Stab_Reg}a), we can see that, for that choice of parameters, the system becomes unstable also when the coupling is symmetric for a sufficiently large $\tau_v$, while in the other three panels (b-d), only through the imaginary part of the coupling operator instability can be reached. This is very important, for instance, in the design of couplings between reactors, as a certain degree of asymmetry may trigger unwanted instabilities.

\section{Conclusions}\label{sec:conclusion}

In this work we have discussed the stability for hyperbolic reaction-diffusion networked systems with asymmetric coupling, showing that achieving stability is a difficult task due to several factors known to trigger instability. Since an asymmetric coupling yields a complex spectrum, i.e., a complex polynomial whose stability needs to be determined, we have carried out the analysis through a recent generalization of the Routh-Hurwitz criterion for complex polynomials. We have analytically studied the stability conditions and we have numerically visualized the stability region for the Brusselator model, a toy model of an autocatalytic chemical reaction, often adopted in the study of reaction-diffusion systems. Nonetheless, let us remark that our results are general, and the analysis can be carried out for any model of the form of Eq. (3). Moreover, the generalization of the Routh-Hurwitz criterion can be used for any polynomial with complex coefficients. 

Lastly, let us point out that, in many real-world systems, the coupling is not only asymmetric, but also non-normal, meaning that the adjacency matrix $A$ satisfies the condition $AA^*\neq A^*A$ \cite{malbor_teo}. In that case, the prediction given by a linear stability analysis may be valid only when the systems is subjected to infinitesimal perturbations. However, in applications the perturbations are always finite and instabilities can arise even if the Routh-Hurwitz conditions are satisfied \cite{jtb}. In such cases, the stability analysis needs to be complemented with simulations, in order to be sure that the system will remain stable against small perturbations. \\

\textit{The authors are grateful to Timoteo Carletti for interesting discussions. R.M. and H.N. acknowledge JSPS KAKENHI JP22K11919, JP22H00516, and JST CREST JP-MJCR1913 for financial support. A.H. is supported by a FNRS Postdoctoral Fellowship, Grant CR 40010909.}


\end{document}